\documentclass[10pt,reqno]{amsart}

\pdfoutput = 1

\usepackage[english]{babel}
\usepackage{amsmath,amssymb,amsthm}
\usepackage{amsfonts}
\usepackage{amsxtra}
\usepackage{array,dcolumn}
\usepackage{graphicx}
\usepackage[hyperfootnotes=true,
        pdffitwindow=true,
        plainpages=false,
        pdfpagelabels=true,
        pdfpagemode=UseOutlines,
        pdfpagelayout=SinglePage,
        hyperindex,]{hyperref}

\usepackage[T1]{fontenc}
\usepackage[utf8]{inputenc}
\usepackage[activate={true,nocompatibility},spacing,kerning]{microtype}
\usepackage[sc,osf]{mathpazo}
\microtypecontext{spacing=nonfrench}

\newcommand{\mathsym}[1]{{}}
\newcommand{\unicode}[1]{{}}

\theoremstyle{plain}
\newtheorem{theorem}{Theorem}

\newtheorem{corollary}[theorem]{Corollary}
\newtheorem{proposition}[theorem]{Proposition}

\theoremstyle{definition}

\theoremstyle{remark}
\newtheorem{remark}[theorem]{Remark}

\newcommand{\R}{\mathbb R}
\newcommand{\C}{\mathbb C}

\newcommand{\half}{\tfrac{1}{2}}

\renewcommand{\geq}{\geqslant}

\begin{document}

%\setlength{\headheight}{20pt}

%\begin{titlepage}
\title[Matrix Polar Decomposition and the Blaschke-Petkantschin Formula]{Matrix Polar Decomposition and Generalisations of the Blaschke-Petkantschin Formula in Integral Geometry}
\author{Peter J. Forrester}
\address{Department of Mathematics and Statistics, University of Melbourne, Victoria 3010, Australia}
\email{pjforr@unimelb.edu.au}
\date{}

%\eads{\mailto{pjforr@unimelb.edu.au}}

\begin{abstract}
In the work [Bull, Austr. Math. Soc. 85 (2012), 315-234], S.~R.~Moghadasi has shown how the decomposition of the $N$-fold product of Lebesgue measure on $\R^n$ implied by matrix polar decomposition can be used to derive the Blaschke-Petkantschin decomposition of measure formula from integral geometry.  We use known formulas from random matrix theory to give a simplified derivation of the decomposition of Lebesgue product measure implied by matrix polar decomposition, applying too to the cases of complex and real quaternion entries, and we give corresponding generalisations of the  Blaschke--Petkantschin formula.
A number of applications to random matrix theory and integral geometry are given, including to the calculation of the moments of the volume content
of the convex hull of $k \le N+1$ points in $\mathbb R^N$, $\mathbb C^N$ or $\mathbb H^N$ with a Gaussian or uniform distribution.
\end{abstract}

%\subjclass[2000]{15A52, 33C45, 33E17, 42C05, 60K35, 62E15}
%\ams{15A52, 33C45, 33E17, 42C05, 60K35, 62E15}
%\keywords{random matrices, eigenvalue distribution, Wishart matrices, Painlev\'e equations, isomonodromic deformations}
%\noindent{\it Keywords\/}: random matrices, eigenvalue distribution, Wishart matrices, Painlev\'e equations, isomonodromic deformations.

%\end{titlepage}
\maketitle

%\tableofcontents

\section{Introduction}\label{s1}
Of fundamental importance to matrix theory is the singular value decomposition. Let $M$ be an $n\times N$ real matrix, and suppose $n\geq N$ so that the columns of $M$ can form a basis of an $N$-dimensional subspace in $\R^n$. Let $U$ ($V$) be an $n\times n$ ($N\times N$) real orthogonal matrix, and let $\Lambda$ be an $n\times N$ diagonal matrix, diagonal entries $\sigma_j$ ($j=1,\ldots,N$) each non-negative. The singular value decomposition takes the form
\begin{align}
M=U\Lambda V^T.\label{1}
\end{align}
Since $M^TM=V\Lambda^T\Lambda V^T$ one sees that the $\{\sigma_j^2\}$ are the eigenvalues of $M^TM$ and $V$ is the corresponding matrix of eigenvectors, while the fact that $MM^T=U\Lambda\Lambda^TU^T$ tells us that the $\{\sigma_j^2\}$ are the non-zero eigenvalues of $MM^T$, and the columns of $U$ are the eigenvectors.

An equivalent way to write \eqref{1} is to first note that
\begin{align*}
\Lambda=\tilde{I}_{n,N}\,\mathrm{diag}(\sigma_1,\ldots,\sigma_N),\quad\tilde{I}_{n,N}:=\begin{bmatrix}I_{N\times N}\\ 0_{(n-N)\times N}\end{bmatrix}.
\end{align*}
This shows
\begin{align}
M=U_1\,\mathrm{diag}(\sigma_1,\ldots,\sigma_N)V^T,\label{2}
\end{align}
where $U_1=U\tilde{I}_{n,N}$; thus $U_1$ consists of the first $N$ columns of $U$. Multiplying $U_1$ in \eqref{2} by $V^TV$ --- which is the identity --- on the right shows that a corollary of the singular value decomposition is the polar decomposition
\begin{align}
M=QP,\label{3}
\end{align}
where $Q=U_1V^T$ has its $N$ columns forming an orthonormal set of vectors in $\R^n$, and $P=V\,\mathrm{diag}(\sigma_1,\ldots,\sigma_N)V^T$ is an $N\times N$ real symmetric matrix.

The decompositions \eqref{2} and \eqref{3} both have analogues for $M$ complex. The singular value decomposition \eqref{2} has the form
\begin{align}
M=U_1\,\mathrm{diag}(\sigma_1,\ldots,\sigma_N)V^{\dagger},\label{4}
\end{align}
where $V$ is an $N\times N$ unitary matrix, and the columns of $U_1$ form an orthonormal set of $N$ vectors in the complex vector space $\C^n$. The polar decomposition \eqref{3} has the form
\begin{align}
M=ZH,\label{5}
\end{align}
where $Z=U_1V^{\dagger}$ also has its $N$ columns forming an orthonormal set of vectors in $\C^n$ and $H=V\,\mathrm{diag}(\sigma_1,\ldots,\sigma_N)V^{\dagger}$ is complex Hermitian.

In a recent paper S.R.~Moghadasi \cite{Mo12} has considered the decompositions of Lebesgue measure implied by \eqref{3} and \eqref{5}. In the real case an application was given to provide a new derivation of a decomposition of Lebesgue measure from integral geometry due originally to Blaschke \cite{Bl35} and Petkantschin \cite{Pe36} (or more precisely, an equivalent form thereof, due to
Miles \cite{Mi71}). A further application was given to the computation of the moments of the determinant of an $N\times N$ standard real Gaussian matrix.

It is an aim of the present paper to give a different derivation to that in \cite{Mo12} of the Lebesgue measure decompositions implied by \eqref{3} and \eqref{5}. This derivation, to be carried out in Section \ref{s2}, is deduced from knowledge of a matrix change of variables formula core to the study of the so-called Wishart matrices in random matrix theory, which in turn is equivalent to a matrix polar integration theorem.
It allows for a natural extension to the case that the entries of $M$ are real quaternions, represented as $2\times 2$ complex matrices
\begin{align}
\begin{bmatrix}z&w\\-\overline{w}&\overline{z}\end{bmatrix}.\label{6}
\end{align}
The matrix polar integration theorem can then, following \cite{Mo12}, be used to derive
Blashke-Petkantschin type decomposition of measure formulas, now for the real, complex and real quaternion cases. This is done in
Section \ref{s3}.

In Section \ref{s4}, a matrix integration formula of Section \ref{s2} is used to deduce the matrix distribution of $U (M^\dagger M)^{1/2}$,
for $U$ a random unitary matrix with entries which are real, complex or real quaternion, corresponding to those of $M$, in
the case that the matrix distribution of $M$ depends on $M^\dagger M$.
As another application,  the evaluation of the moments of  determinants associated with certain classes of random matrices 
with entries from these number fields is given, as is an affine version of this result.
In the real case, the affine version relates to the volume of the convex hull of $k \le N + 1$ points
uniformly distributed in the unit ball in $\mathbb R^N$, or on its surface, as first calculated by Kingman \cite{Ki69} and Miles \cite{Mi71}.

\section{Decomposition of Measure Implied by Matrix Polar Decomposition}\label{s2}
In the real case the singular value and polar decompositions are given by \eqref{2} and \eqref{3}, while in the complex case they are given by \eqref{4} and \eqref{5}. Also of interest is the real quaternion case, in which the entries of $M$ have the $2\times2$ block structure \eqref{6}. Since there are in general four independent real entries in \eqref{6}, we will annotate the matrices in \eqref{4} and \eqref{5} to involve superscripts $4$ so that for the singular value decomposition we have
\begin{align}
M^{(4)}=U^{(4)}_1\,\mathrm{diag}(\sigma_1\mathbb{I}_2,\ldots,\sigma_N\mathbb{I}_2)V^{(4) \, {\dagger}},\label{7}
\end{align}
and for the polar decomposition we have
\begin{align}
M^{(4)}=Z^{(4)}H^{(4)}, \label{8}
\end{align}
As a complex matrix, $M^{(4)}$ is of size $2n\times 2N$. The matrix $V^{(4)}$ is a $2N\times 2N$ unitary matrix with $2\times2$ real quaternion entries, and thus after conjugation by a permutation matrix is a member of the classical group $\mathrm{Sp}(2N)$. Each block column of $U^{{(4)}}_1$ can be regarded as a member of the real quaternion vector space $\mathbb{H}^n$, and the $N$ block columns together form an orthonormal set. Note that the singular values in \eqref{7} are doubly degenerate. This is in keeping with $W^{(4)}:=M^{(4) \, \dagger}M^{(4)}$ having the symmetry $W^{(4)}=\tilde{Z}_{2N}\overline{W}^{(4)}\tilde{Z}_{2N}^{-1}$, where
\begin{align*}
\tilde{Z}_{2N}:=\mathbb{I}_N\otimes\begin{bmatrix}0&-1\\1&0\end{bmatrix},
\end{align*}
which implies the eigenvalues of $W^{(4)}$ are doubly degenerate. In \eqref{8} $Z_4=U_1^{(4)}V^{(4) \, \dagger}$, and this $2n\times2N$ matrix has the same properties as those highlighted for $U_{1,4}$. Also $H^{(4)}=V^{(4)}\,\mathrm{diag}(\sigma_1\mathbb{I}_2,\ldots,\sigma_N\mathbb{I}_2)V^{(4) \, \dagger}$, which is a $2N\times 2N$ Hermitian matrix with real quaternion blocks \eqref{6}.

Introduce the notation $(\mathrm{d}M)$ for the product of the independent differentials (real and imaginary parts) of the entries of the matrix $M$. Let $\beta=1$, $2$ or $4$ according to the entries of $M$ being real, complex or real quaternion. As a first use of this notation, we write $\mathbb{F}_{\beta}$ to denote these number fields. For $M$ as appearing in \eqref{3}, \eqref{5} and \eqref{8}, $(\mathrm{d}M)$ consists of $\beta nN$ independent differentials.

Consider the matrix $W=M^{\dagger}M$. In the real case it follows from \eqref{5} that $W=H^2$; in the real quaternion case it follows from \eqref{8} that $W=(H^{(4)})^2$. In mathematical statistics a change of variables between the entries of $M$ and the entries of $W$ was first studied by Wishart \cite{Wi28} in the real case. This was further refined by James \cite{Ja54} (see the text book treatment in \cite{Mu82}) to obtain a change of variables formula associated with both \eqref{2} and \eqref{3}. The analogue of this in the complex case is detailed in \cite[\S 3.2.3]{Fo10}, and the modification required for the real quaternion case given in 
\cite[Ex.~3.2 q.5]{Fo10}. The change of variables formula in all three cases is conveniently summarised in \cite[Prop.~4]{DG11} and reads
\begin{align}
(\mathrm{d}M)=2^{-N}(\mathrm{det}\,W)^{\beta(n-N+1)/2-1}\,(\mathrm{d}W)\,\left(U_1^{\dagger}\mathrm{d}U_1\right).\label{9}
\end{align}
Here $\left(U_1^{\dagger}\mathrm{d}U_1\right)$ is the invariant measure on the set of $n\times N$ matrices with elements in $\mathbb{F}_{\beta}$, and having the property that the $N$ columns form an orthonormal set in the vector space $\mathbb{F}_{\beta}^n$. These matrices are said to define the Stiefel manifold, to be denoted $\mathcal{V}_{N,n}^{\beta}$ so we have $U_1\in\mathcal{V}_{N,n}^{\beta}$ and $\left(U_1^{\dagger}\mathrm{d}U_1\right)$ is the invariant measure on this manifold. Also, in terms of the singular values, we define
\begin{align}
\det W=\prod_{l=1}^N\sigma_l^2\label{10}
\end{align}
in all three cases. Thus the usual definition of the determinant is modified in the case $\beta=4$, where $W$ has eigenvalues 
$\{\sigma_l^2\}$ each with multiplicity $2$; in (\ref{10}) the multiplicity is ignored.

From \eqref{9} we can read off a matrix polar integration theorem, reproducing the main theorem, Theorem 2.5, of \cite{Mo12} which relates to the real and complex cases, and extending it to the real quaternion case.

\begin{proposition}
Let $\mathcal{M}_{n\times N}^{\beta}$ denote the set of $n\times N$ matrices with entries in $\mathbb{F}_{\beta}$. For $M\in\mathcal{M}_{n\times N}^{\beta}$ introduce the polar decomposition according to
\begin{align*}
M=U_1W^{1/2},
\end{align*}
where $U_1\in\mathcal{V}_{N,n}^{\beta}$ and $W=M^{\dagger}M\in\mathcal{P}_{+,N}^{\beta}$, $\mathcal{P}_{+,N}^{\beta}$ denoting the set of Hermitian matrices with entries in $\mathbb{F}_{\beta}$ and having all eigenvalues non-negative. For $g:\mathcal{M}_{n\times N}^{\beta}\rightarrow\R$ absolutely integrable we have
\begin{align}
\int_{\mathcal{M}_{n\times N}^{\beta}}\,g(M)\,\mathrm{d}M=2^{-N}\int_{\mathcal{V}_{N,n}^{\beta}}\,\left(U_1^{\dagger}\mathrm{d}U_1\right)\,\int_{\mathcal{P}_{+,N}^{\beta}}\,(\mathrm{d}W)\,\left(\mathrm{det}\,W\right)^{\beta(n-N+1)/2-1}\nonumber
\\\times g(U_1W^{1/2}).\label{11}
\end{align}
\end{proposition}

\begin{remark}
In \cite{Mo12} the factor $2^{-N}$ on the RHS is replaced by $2^{-N^2}$ in the complex case and by $2^{-N(N+1)/2}$ in the real case, due to a different normalisation of  $\left(U_1^{\dagger}\mathrm{d}U_1\right)$. The normalisation of $\left(U_1^{\dagger}\mathrm{d}U_1\right)$ inherent in \eqref{11} can readily be determined by using Gram-Schmidt orthogonalisation to write
\begin{align}
M=\tilde{U}_1T,\label{12}
\end{align}
where $\tilde{U}_1\in\mathcal{V}_{N,n}^{\beta}$ and $T$ is an $N\times N$ upper triangular matrix with non-negative diagonal entries $\{t_{ii}\}_{i=1}^N$ and strictly upper triangular entries in $\mathbb{F}_{\beta}$ (in the case $\beta=4$, $T$ is represented as a $2N\times2N$ complex matrix with diagonal entries $t_{ii}\mathbb{I}_2$). A standard change of variables formula in random matrix theory tells us (see e.g. \cite{Mu82} for the real case, \cite{Fo10} or \cite[Lemma 2.7]{DG09} for all three cases)
\begin{align}
(\mathrm{d}M)=\prod_{j=1}^N t_{jj}^{\beta(n-j+1)-1}\,(\mathrm{d}T)\,\left(\tilde{U}_1^{\dagger}\mathrm{d}\tilde{U}_1\right).\label{13}
\end{align}
Multiplying both sides by $e^{-\mathrm{Tr}\,M^{\dagger}M}=e^{-\mathrm{Tr}\,T^{\dagger}T}$, taking $\mathrm{Tr}$ in the real quaternion case to be $\half$ of its usual meaning in keeping with (\ref{10}), and integrating gives
\begin{align}
\int_{U_1\in\mathcal{V}_{N,n}^{\beta}}\,\left(U_1^{\dagger}\mathrm{d}U_1\right)=2^N\prod_{i=1}^N\frac{\pi^{\beta(n-i+1)/2}}{\Gamma\left(\beta(n-i+1)/2\right)}=\prod_{i=1}^N\sigma_{\beta(n-i+1)},\label{14}
\end{align}
where $\sigma_l=2\pi^{l/2}/\Gamma(l/2)$ is the surface area of the $(l-1)$-sphere in $\R^l$. In the real case the above calculation can be found in \cite[Ch.~2]{Mu82}.
\end{remark}

We can make use of \eqref{14} to simplify \eqref{11} in the case that $g(M)$ is independent of $U_1$. This extends matrix integration formulas from \cite[Th.~2.5]{Mo12}, applicable for real or complex entries, to the real quaternion case.

\begin{corollary}\label{C3}
For $f:\mathcal{P}_{+,N}^{\beta}\rightarrow\R$ absolutely integrable we have
\begin{multline}
\int_{\mathcal{M}_{n\times N}^{\beta}}\,f(M^{\dagger}M)\,(\mathrm{d}M)
\\=\prod_{i=1}^N\frac{\sigma_{\beta(n-i+1)}}{\sigma_{\beta(N-i+1)}}\,\int_{\mathcal{M}_{N\times N}^{\beta}}\,f(M^{\dagger}M)\,\left(\mathrm{det}\,M^{\dagger}M\right)^{\beta(n-N)/2}(\mathrm{d}M).\label{15}
\end{multline}
\end{corollary}

\section{Blaschke-Petkantschin Formula}\label{s3}
\subsection{Blaschke-Petkantschin Formula --- Miles Version}\label{s3.1}
The Grassmannian $G_{N,n}^{\beta}$ is the set of all $N$-dimensional subspaces in $\mathbb{F}_{\beta}^n$. This set can be identified with the quotient space $\mathcal{V}_{N,n}^{\beta}/\mathrm{O}(N)$. Thus, since the elements of $\mathcal{V}_{N,n}^{\beta}$ are $N$ orthonormal vectors in $\mathbb{F}_{\beta}^n$, elements in the quotient consist of all $N$ orthonormal vectors spanning the same $N$-dimensional subspace. It follows that we have the decomposition of measure (see \cite{Ja54} for the real case)
\begin{align}
\left(\tilde{U}_1^{\dagger}\mathrm{d}\tilde{U}_1\right)=\mathrm{d}\omega_{N,n}^{\beta}\,\left(U^{\dagger}\mathrm{d}U\right),\label{17}
\end{align}
where $\tilde{U}_1\in\mathcal{V}_{N,n}^{\beta}$, $U\in\mathcal{V}_{N,N}^{\beta}$ and $\mathrm{d}\omega_{N,n}^{\beta}$ denotes the invariant measure on $G_{N,n}^{\beta}$. As a matrix decomposition
\begin{align}
\tilde{U}_1=AU\label{18}
\end{align}
for some $A$ such that the columns of $A$ form a "reference" set of orthogonal vectors spanning the particular $N$-dimensional subspace of $\mathbb{F}_{\beta}^n$. For future application we note from (\ref{17}) and (\ref{14}) that
\begin{equation}\label{18'}
{\rm vol} \, G_{N,n}^\beta =
{{\rm vol} \, {\mathcal{V}_{N,n}^{\beta}} \over {\rm vol} \, {\mathcal{V}_{N,N}^{\beta}}} =
\prod_{i=1}^N\frac{\sigma_{\beta(n-i+1)}}{\sigma_{\beta(N-i+1)}}.
\end{equation}

Making use of \eqref{17} and \eqref{18}, it was observed by Moghadasi \cite{Mo12} that in the real case \eqref{11} can be used to deduce an integration formula equivalent to the Blaschke-Petkantschin decomposition of measure formula from integral geometry \cite{Bl35}, 
\cite{Pe36}, 
\cite{BN89}, in form due to Miles \cite{Mi71}, . However the required working is not restricted to the real case, allowing for a generalisation of the Blaschke-Petkantschin formula to the complex and real quaternion cases.

\begin{proposition}
Let $g:\mathcal{M}_{N,n}^{\beta}\rightarrow\R$ be integrable. Make use of the notations $G_{N,n}^{\beta}$, $\mathrm{d}\omega_{N,n}^{\beta}$ for the Grassmannian and corresponding invariant measure as introduced above. We have
\begin{multline}
\int_{M\in\mathcal{M}_{N,n}^{\beta}}\,g(M)\,(\mathrm{d}M)
\\=\int_{A\in G_{N,n}^{\beta}}\,\mathrm{d}\omega_{N,n}^{\beta}\,\int_{M\in\mathcal{M}_{N,N}^{\beta}}\,(\mathrm{d}M)\,g(AM)\,\left(\mathrm{det}\,M^{\dagger}M\right)^{\beta(n-N)/2}.\label{19}
\end{multline}
\end{proposition}
\begin{proof}
Substituting \eqref{17} and \eqref{18} in the RHS of \eqref{11} this reads
\begin{multline}
\int_{\mathcal{M}_{n\times N}^{\beta}}\,g(M)\,(\mathrm{d}M)=2^{-N}\int_{A\in G_{N,n}^{\beta}}\,\mathrm{d}\omega_{N,n}^{\beta}\,\int_{U\in\mathcal{V}_{N,N}^{\beta}}\,\left(U^{\dagger}\mathrm{d}U\right)
\\ \times\int_{W\in\mathcal{P}_{+,N}^{\beta}}\,(\mathrm{d}W)\,g(AUW^{1/2})\,\left(\mathrm{det}\,W\right)^{\beta(n-N+1)/2-1}.\label{20}
\end{multline}
On the other hand, in the case $n=N$, and with $g(M)$ replaced by 
$$
g(AM)\,\left(\mathrm{det}\,M^{\dagger}M\right)^{\beta(n-N)/2},
$$
these same substitutions in \eqref{11} tell us that
\begin{multline}
\int_{\mathcal{M}_{N\times N}^{\beta}}\,g(AM)\,\left(\mathrm{det}\,M^{\dagger}M\right)^{\beta(n-N)/2}\,(\mathrm{d}M) \\
= 2^{-N} \int_{U \in\mathcal{V}_{N,N}^{\beta}}\,\left(U^{\dagger}\mathrm{d}U\right)
\int_{W\in\mathcal{P}_{+,N}^{\beta}}\,(\mathrm{d}W)\,g(AUW^{1/2})\,\left(\mathrm{det}\,W\right)^{\beta(n-N+1)/2-1}.\label{21}
\end{multline}
Substituting (\ref{21}) in (\ref{20}) gives (\ref{19}).
\end{proof}

\subsection{Blaschke-Petkantschin Formula --- Affine Version}\label{s3.1a}
One viewpoint on (\ref{19}) is as the decomposition of measure
\begin{equation}\label{19'}
\prod_{k=1}^N \mathrm{d} \mathbf v_k^n = \Big | \det [ \mathbf v_k^N ]_{k=1}^N \Big |^{\beta ( n - N)}
\prod_{k=1}^N \mathrm{d} \mathbf v_k^N \mathrm{d} \omega_{N,n}^\beta,
\end{equation}
where $ \mathbf v_k^n \in (\mathbb F_\beta)^n$, while $ \mathbf v_k^N \in  (\mathbb F_\beta)^N$ is the
co-ordinate for $ \mathbf v_k^n$ in a particular basis for ${\rm Span} \, \{  \mathbf v_k^n \}_{k=1}^N$, and
$d \omega_{N,n}^\beta$ is the invariant measure on subspaces (the Grassmannian $G_{N,n}^\beta$) formed by the span.
Following \cite{Mi71}, consider $\mathbf v_k^n \in (\mathbb F_\beta)^n$ for $k=0,1,\dots,N$, and define
$\mathbf z_k^n = \mathbf v_k^n- \mathbf v_0^n$ $(k=1,\dots,N)$. Then we have
\begin{align}\label{19.1}
\prod_{k=0}^N \mathrm{d} \mathbf v_k^n & = d \mathbf v_0^n \prod_{k=1}^N \mathrm{d} \mathbf z_k^n \nonumber \\
& =  \mathrm{d} \mathbf v_0^n  \Big | \det [ \mathbf z_k^N ]_{k=1}^N \Big |^{\beta (n - N)}
\prod_{k=1}^N \mathrm{d} \mathbf z_k^N \mathrm{d} \omega_{N,n}^\beta, 
\end{align}
where the second equality follows upon applying (\ref{19'}) to $\{  z_k^n \}_{k=1}^N$. With $B$ an
$n \times N$ matrix with columns forming an orthogonal set which is a basis for Span$\, \{ \mathbf v_k^N - \mathbf v_0^N \}_{k=1}^N$,
we have $\mathbf z_k^n = B \mathbf z_k^N$. Now define $\mathbf z_k^N = \mathbf v_k^N -
\mathbf v_0^N$. Then (\ref{19.1}) can be rewritten
\begin{equation}\label{19.2}
\prod_{k=0}^N \mathrm{d} \mathbf v_k^n =  \mathrm{d} \mathbf v_0^n 
\Big | \det [ \mathbf v_k^N - \mathbf v_0^N ]_{k=1}^N \Big |^{\beta ( n - N)}
\prod_{k=1}^N \mathrm{d} \mathbf v_k^N \, \mathrm{d} \omega_{N,n}^\beta.
\end{equation}
Finally, write $\mathbf v_0^n = B \mathbf v_0^N + \mathbf r$, where $\mathbf r$ is an element of the orthogonal complement
of the column space of $B$, and let $\mathrm{d} S_{n-N}^{\perp,\beta}$ denote the volume element associated with this subspace.
This reduces (\ref{19.2}) to the affine form of (\ref{19'}).

\begin{proposition}
We have 
\begin{equation}\label{19.3}
\prod_{k=0}^N \mathrm{d} \mathbf v_k^n =  
\Big | \det [ \mathbf v_k^N - \mathbf v_0^N ]_{k=1}^N \Big |^{\beta ( n - N)}
\prod_{k=0}^N \mathrm{d} \mathbf v_k^N \, \mathrm{d} \omega_{N,n}^\beta \, \mathrm{d} S_{n-N}^{\perp,\beta}.
\end{equation}
\end{proposition}

The formula (\ref{19.3}) is the decomposition of measure given in the original works \cite{Bl35} and \cite{Pe36}. With
$B$ and $\mathbf r$ as specified in the above text, we note that $\mathbf v_k^n = B \mathbf v_k^N + \mathbf r$, and
in particular
\begin{equation}\label{19.4}
\| \mathbf v_k^n \| = \| \mathbf v_k^N \|^2 +  \|  \mathbf r \|^2.
\end{equation}
Also, the fact that $B^\dagger B = I_N$ allows us to read off from (\ref{19.3}) an analogue of the matrix integration formula
(\ref{15}).

\begin{corollary}\label{C3.6}
For $f:\mathcal{P}_{+,N}^{\beta}\rightarrow\R$ absolutely integrable, and let $V_n = [ \mathbf v_k^n - \mathbf v_0^n]$. We have
\begin{multline}
\int_{\mathbf v_k^n \in (\mathbb F_\beta)^n}\,f(V_n^{\dagger}V_n)\, \prod_{k=0}^N \mathrm{d} \mathbf v_k^n
\\=\prod_{i=1}^N\frac{\sigma_{\beta(n-i+1)}}{\sigma_{\beta(N-i+1)}}\,\int_{\mathbf v_0^n \in (\mathbb F_\beta)^n, \mathbf v_k^N \in (\mathbb F_\beta)^N}\,f(V_N^{\dagger}V_N)\,| \det V_N |^{\beta(n-N)} \,  \mathrm{d} \mathbf v_0^n
\prod_{k=1}^N \mathrm{d} \mathbf v_k^N  .\label{15p}
\end{multline}
\end{corollary}

\section{Applications}\label{s4}
\subsection{Induced Random Matrix Ensembles}
Suppose the random matrix $M \in \mathcal M_{n \times M}^\beta$ has PDF of the form $P(M^\dagger M)$. The
corresponding induced ensemble is defined as the set of random matrices of the form $K = U (M^\dagger M)^{1/2}$,
where $U \in \mathcal{V}_{N,N}^{\beta}$. It is known \cite{FBKSZ12} that the PDF of $K$ is proportional to 
$(\det K^\dagger K)^{\beta(n-N)/2} P(K^\dagger K)$. Here it will be shown that this result, together with the
proportionality constant, can be deduced as a consequence of Corollary \ref{C3}.

\begin{corollary}
Let $M, U$ and $K$ be specified as in the above paragraph. The PDF of $K$ is equal to 
\begin{equation}\label{22}
\prod_{i=1}^N\frac{\sigma_{\beta(n-i+1)}}{\sigma_{\beta(N-i+1)}}\,
(\det K^\dagger K )^{\beta (n - N)/2} P(K^\dagger K).
\end{equation}
\end{corollary}

\begin{proof}
From the definitions the PDF of $K$ is equal to
\begin{equation}\label{23}
\int_{\mathcal{V}_{N,N}^{\beta}} [U^\dagger {\rm d} U] \int_{ \mathcal M_{n \times N}^\beta} (d M) \,
\delta(K - U(M^\dagger M)^{1/2}) \, P(M^\dagger M),
\end{equation}
where $ [U^\dagger {\rm d} U]$ denotes the normalised invariant measure on $\mathcal{V}_{N,N}^{\beta}$. Making use of
(\ref{15}) with $f(M^\dagger M) = \delta(K - U(M^\dagger M)^{1/2}) \, P(M^\dagger M)$ this is equal to
\begin{multline}\label{24}
\prod_{i=1}^N\frac{\sigma_{\beta(n-i+1)}}{\sigma_{\beta(N-i+1)}}\,
\int_{\mathcal{V}_{N,N}^{\beta}} [U^\dagger {\rm d} U] \int_{ \mathcal M_{n \times N}^\beta} (d M) \,
\delta(K - U(M^\dagger M)^{1/2}) \\
\times \left(\mathrm{det}\,M^{\dagger}M\right)^{\beta(n-N)/2}  P(M^\dagger M).
\end{multline}
From the matrix polar decomposition we know that for a unique $U \in \mathcal{V}_{N,N}^{\beta}$ and with
$M \in { \mathcal M_{N \times N}^\beta}$, we have $M = U(M^\dagger M)^{1/2}$. Hence in (\ref{24}) (but not in
(\ref{23}) we can replace the delta function by $\delta (K - M)$. The integration over $U$ thus decouples, giving unity,
while the integration over $M$ gives (\ref{20}).
\end{proof}

\subsection{Moments of Determinants}
In the real case, a well established application of the Blascke-Petkantschin formula (\ref{19}) is to the evaluation
of the moments of $\det M^\dagger M$ in the case that the PDF for the distribution of $M$ is a function of the
lengths $\| \mathbf m_k \|$, where $ \mathbf m_k$ denotes the $k$-th column of $M$
(see \cite{Mi71}, \cite[\S 3.1.3]{Ma99}). Since $A$ in (\ref{19}) has the property that $A^\dagger A = I_N$, and since the
$k$-th column of $AM$ is $A \mathbf m_k$, we see that choosing $g(M)$ to be a function of the lengths of the columns
we have  $p(M) = p(AM)$. Making use too of (\ref{18'}) shows that in this circumstance the
 Blaschke--Petkantschin formula reduces to
 \begin{multline}
\int_{\mathcal{M}_{n\times N}^{\beta}}\,g(M)\,(\mathrm{d}M)
\\=\prod_{i=1}^N\frac{\sigma_{\beta(n-i+1)}}{\sigma_{\beta(N-i+1)}}\,\int_{\mathcal{M}_{N\times N}^{\beta}}\,g(M)\,\left(\mathrm{det}\,M^{\dagger}M\right)^{\beta(n-N)/2}(\mathrm{d}M)\label{25}
\end{multline}
(cf.~(\ref{15})).

A simple example depending only on the lengths of the columns is when
\begin{equation}\label{26}
g(M) = \prod_{i=1}^N {p}(\| \mathbf m_i \|).
\end{equation}
Let us normalise $g(M)$ for $M \in \mathcal M_{N \times N}$ by requiring that
$$
\int_{\mathbb (F_\beta)^N} {p}( \| \mathbf m \| ) \, \mathrm d  \mathbf m =
\sigma_{\beta N} \int_0^\infty r^{\beta N - 1} {p}(r) \, \mathrm dr = 1,
$$
where the first equality follows by converting to polar coordinates. So if we choose ${p}(r)$ to be  uniform
for $0 < r < 1$ and zero for $r > 1$, to be correctly normalised we should set
\begin{equation}\label{27}
{p}(r) = {\beta N \over \sigma_{\beta N}} \chi_{0 < r < 1},
\end{equation}
where $\chi_S = 1$ for $S$ true and $\chi_S = 0$ otherwise. And requiring each $\mathbf m_i$ to have length unity is achieved
by choosing
\begin{equation}\label{27'}
{p}(r) = {1 \over \sigma_{\beta N}} \delta(1 - \| \mathbf m \|).
\end{equation}
Also of interest is the Gaussian distribution
\begin{equation}\label{27a}
P(M) = \Big ( {1 \over \pi} \Big )^{\beta N^2/2} e^{- {\rm Tr} \, M^\dagger M} :=
\Big ( {1 \over \pi} \Big )^{\beta N^2/2} e^{- \sum_{l=1}^N \| \mathbf m_l \|^2 }
\end{equation}
(this in the real quaternion case Tr is defined as ${1 \over 2}$ its usual value; cf.~(\ref{10})) for which
\begin{equation}\label{27b}
{p}(r) = \Big ( {1 \over \pi} \Big )^{\beta N/2} e^{- r^2}.
\end{equation}

The corresponding moments of $\det M^\dagger M$ can be
read off from (\ref{25}).

\begin{proposition}\label{P7}
Let the PDF of the $N \times N$ matrix $M$ be of the form (\ref{26}) with ${p}(r)$ specified by (\ref{27}). We have, for 
$\mathbf m_i \in (\mathbb F_\beta)^N$, $(i=1,\dots,N)$, and $q$ such that both sides are well defined
 \begin{multline}\label{35}
\Big ( {\beta N \over \sigma_{\beta N}} \Big )^N
\int_{ \| \mathbf m_i \| \le 1}
| \det M |^{q} \, (\mathrm d M) =
\prod_{i=1}^N\frac{\sigma_{\beta(N-i+1)}}{\sigma_{\beta(n-i+1)}}\,
\Big ( {N  \sigma_{\beta n} \over n  \sigma_{\beta N}} \Big )^{N } \bigg |_{n = N + q/\beta}
\end{multline}
and
 \begin{multline}\label{35'}
\Big ( {1 \over \sigma_{\beta N}} \Big )^N
\int_{ \| \mathbf m_i \| = 1}
| \det M |^{q} \, (\mathrm d M) =
\prod_{i=1}^N\frac{\sigma_{\beta(N-i+1)}}{\sigma_{\beta(n-i+1)}}\,
\Big ( {  \sigma_{\beta n} \over   \sigma_{\beta N}} \Big )^{N } \bigg |_{n = N + q/\beta}.
\end{multline}
In the Gaussian case (\ref{27a}),
 \begin{multline}\label{35c}
 \Big ( {1 \over \pi} \Big )^{\beta N^2/2}
 \int_{(\mathbb F_\beta)^N}   e^{- \sum_{l=1}^N \| \mathbf m_l \|^2 }  | \det M |^{q}  \, (\mathrm d M) \\
 = 
  \prod_{i=1}^N\frac{\sigma_{\beta(N-i+1)}}{\sigma_{\beta(n-i+1)}}\,
\Big ( {1 \over \pi} \Big )^{\beta N (N - n)/2}  \Big |_{n = N + q/\beta}.
\end{multline}

\end{proposition}
\begin{proof}
The only point that remains is to justify extending from values $q = (n - N)\beta$ with $n \in \mathbb Z_{\ge 0}$ to the
general values of $q$. This can be done by appealing to Carlson's theorem (see e.g.~\cite[Prop,~4.1.4]{Fo10}), which
tells is that if two functions, analytic in the right half plane, agree at the integers and are bounded by
${\rm O}(e^{\mu |z|})$, $\mu < \pi$, in this region then the functions are identical.
\end{proof}

\begin{remark} In the real case (\ref{27a}) corresponds to a result of Wilks \cite{Wi32}; the method of derivation presented
here is that of \cite{Mo12}. For each $\beta = 1,2$ and 4 both (\ref{35'}) and (\ref{35c}) can be found in \cite{Ro07},
where the derivation made use the Jacobian for a QR type decomposition; see also \cite{Ma98,Ma99}.
Introducing polar coordinates $\mathbf m_i = r_i \mathbf q_i$ when $\| \mathbf  q_i \| = 1$ in  (\ref{36}), the dependence on
$r_i$ can be integrated out and, after use of the explicit formula for $\sigma_l$ below (\ref{14}),
(\ref{35'}) reclaimed. This same procedure can be used in (\ref{35}) to provide another derivation of (\ref{35'}).
\end{remark}

\begin{remark} Differentiating both sides of (\ref{35})--(\ref{27a}) with respect to $q$ and setting
$q=0$ provides us with the evaluation of $\langle \log | \det M | \rangle$. In the Gaussian case, this average has appeared
before, in the context of calculating the Lyapunov exponents for the random matrix product
\begin{equation}\label{LG}
L_t = M^{(t)} M^{(t-1)} \cdots M^{(1)},
\end{equation}
where each $M^{(j)}$ is chosen from the distribution of $M$ \cite{Ne86,Fo13,Fo15}. In relation to this problem, one recalls from
the multiplicative ergodic theorem of Oseledec that the limiting matrix $V_N := \lim_{t \to \infty} (L_t^\dagger L_t)^{1/(2t)}$ is well
defined. The eigenvalues of $V_N$ are all non-negative, and when written in the form $\{e^{\mu_l} \}_{l=1}^N$, the
exponents $\mu_l$ are the Lyapunov exponents. Let $\mathbf u$ be a unit vector. In the case that each column of $M^{(j)}$
in (\ref{LG}) depends only on its length,  the distribution of $\mathbf u^\dagger M^{(j)}$ is independent of $\mathbf u$, and a
result of Raghunathan  relating to the partial sums of the Lyapunov exponents tells us that in this setting
\begin{equation}\label{LG1}
\sum_{i=1}^N \mu_i = \langle \log | \det M | \rangle.
\end{equation}
\end{remark}

The results of Proposition \ref{P7} when used in the RHS of (\ref{15}) provide us with an extension of the moment
formulas to now apply to $\det M^\dagger M$ when $M = [\mathbf m_k^n]_{k=1}^N$ is an $n \times N$ matrix with entries
with entries in $\mathbb F_\beta$. In the real case, this was first noticed and made explicit in \cite{Mi71}.

\begin{corollary}\label{C3.8}
With $M$ specified as above, let $M_N = [\mathbf m_k^N]_{k=1}^N$ correspond to the $N \times N$ case.
Let the notation $\langle \cdot \rangle_\#$ refer to the expectation with respect to the PDF relating to $\#$.
Set $A_{\beta,n,N} = \prod_{i=1}^N \sigma_{\beta(n-i+1)} / \sigma_{\beta(N-i+1)}$. We have
\begin{align*}
\Big \langle ( \det M^\dagger M)^{h/2}  \Big  \rangle_{\| \mathbf m_k^n \| \le 1} & =
A_{\beta,n,N} \Big ( {n \sigma_{\beta N} \over N \sigma_{\beta n} }   \Big )^N 
\Big \langle | \det M_N |^{h + \beta(n-N)}  \Big  \rangle_{\| \mathbf m_k^N \| \le 1}  \\
\Big \langle ( \det M^\dagger M)^{h/2}  \Big  \rangle_{\| \mathbf m_k^n \| = 1} & =
A_{\beta,n,N} \Big ( { \sigma_{\beta N} \over  \sigma_{\beta n} }    \Big )^N 
\Big \langle | \det M_N |^{h + \beta(n-N)} \Big  \rangle_{\| \mathbf m_k^N \| = 1} \\
\Big \langle ( \det M^\dagger M)^{h/2} \Big \rangle_{{\rm G}^n} & =
A_{\beta,n,N}
\Big ( {1 \over \pi} \Big )^{\beta N (n - N)/2}  \Big \langle | \det M_N |^{h + \beta(n-N)} \Big \rangle_{{\rm G}^N},
\end{align*}
where in the final equation the subscripts ${\rm G}^n$ and ${\rm G}^N$ refer to the Gaussian distribution on $M$ and $M_N$.
\end{corollary}

\begin{remark} Generally, for a bounded convex set $K \in \mathbb R^N$, and with $B_N$ the unit $N$-ball, Steiner's
formula gives a polynomial expansion for the volume of the particular smoothing of $K$
defined by Minkowski addition according to $K + r B_N$. Thus \cite{Sc14}
$$
{\rm vol}_N(K + r B_N) = \sum_{k=0}^N r^{N-k} \kappa_{N-k} V_k(K), \qquad r>0,
$$
where $\kappa_k = {1 \over k} \sigma_k$ is the volume of $B_k$,
and $V_k(K)$ is the so-called $k$-th intrinsic volume of $K$, given by
 \begin{equation}\label{CD1}
 V_k(K) = \binom{N}{k} {\kappa_N \over \kappa_k \kappa_{N-k}}
 \int_{A \in G_{k,N}^{\beta = 1}}
 \Big \langle {\rm vol}_k ( \Pi_A (K) \Big \rangle_{\Pi_A \in G_{k,N}} \, {\rm d} \omega_{N,n}^{\beta = 1},
 \end{equation}
 with $\Pi_A(K)$ denoting the orthogonal projection of $K$ onto the random subspace $A$. With $K$ the parallelepiped
 formed by the columns of an $N \times N$ random matrix $M$, and with the distribution of the columns of $M$ dependent only
 on the lengths, we have
   \begin{equation}\label{CD2}
 \Big  \langle  \Big  \langle   {\rm vol}_k ( \Pi_A (K) \Big \rangle_{\Pi_A \in G_{k,N}^{\beta = 1}}  \Big \rangle_K= 
  \Big  \langle ( \det M_k^T M_k)^{1/2}  \Big \rangle_{M_k},
  \end{equation}
where $M_k$ is the $N \times k$ truncation of $M$ to the first $k$ columns. The average on the RHS of (\ref{CD2}) can
be read off from Corollary \ref{C3.8} with $k=1$, $\beta = 1$, $n \mapsto N$ and $N \mapsto k$. Substituting in
(\ref{CD1}) gives, for the Gaussian case,
  \begin{equation}\label{CD3}
  \Big  \langle   V_k(K)  \Big \rangle_{G^{\rm N}} = \binom{N}{k} {\kappa_N \over \kappa_k \kappa_{N-k}}  \pi^{k/2} {\sigma_{N-k+1} \over \sigma_{N+1}}.
   \end{equation}
   \end{remark}

\subsection{Moments of Determinants --- An Affine Setting}
In the real case $D:=| \det M |$ is equal to the volume of the parallelpiped formed by the
columns of $M$, or equivalently is $N!$ times the volume of the simplex with vertices formed by columns of
$M$ together with the origin. It is an elementary fact that $D$ with
$\tilde M = [ \mathbf m_k^N - \mathbf m_0^N ]_{k=1}^N$ is equal to $N!$ times  the volume of the simplex with vertices
given by $\{\mathbf m_k^N \}_{k=0}^N$ \cite{Ma99}. This geometric interpretation has motivated a number of
studies into the moments of $D$ with $\tilde M$ of this structure and $\{\mathbf m_k^N \}_{k=0}^N$ chosen according to
certain probability distributions
 \cite{Ki69,Mi71,RM79,Ma99}. The primary tool is the affine Blascke--Petkantschin decomposition of measure (\ref{19.3})
 in the real case. Following the method of derivation from \cite{Mi71}, we can extend these computations to the
 complex and real quaternion cases, thus obtaining an affine analogue of Proposition \ref{P7}.

\begin{proposition}\label{P8}
Let the $N \times N$ matrix $M$ be defined in terms of vectors $\{\mathbf m_k^N \}_{k=0}^N$,
$m_k^N \in  (\mathbb F_\beta)^N$, $(k=0,\dots,N)$, by $\tilde M = [ \mathbf m_k^N - \mathbf m_0^N ]_{k=1}^N$.
Define the Euler beta function by
 \begin{equation}\label{Beta}
 B(x,y) = \int_0^1 t^{x-1} (1 - t)^{y-1} \, \mathrm dt = {\Gamma (x) \Gamma(y) \over \Gamma (x + y)}.
 \end{equation}
For values of $q$ such that both sides are well defined we have
 \begin{multline}\label{36}
\Big ( {\beta N  \over \sigma_{\beta N}} \Big )^{N+1}
\int_{ \| \mathbf m_i^N \| \le 1}
|\det \tilde M|^{q} \, (\mathrm d \mathbf m_0^N) \cdots  (\mathrm d \mathbf m_N^{N}) \\ = {2 \over  \sigma_{\beta (n - N)} B({\beta \over 2}(n-N), 
{\beta \over 2} N(n+1) + 1)}
\prod_{i=1}^N\frac{\sigma_{\beta(N-i+1)}}{\sigma_{\beta(n-i+1)}}\,
\Big ( {N  \sigma_{\beta n} \over n  \sigma_{\beta N}} \Big )^{N +1} \bigg |_{n = N + q/\beta},
\end{multline}
\begin{multline}\label{36'}
\Big ( {1  \over \sigma_{\beta N}} \Big )^{N+1}
\int_{ \| \mathbf m_i^N \| = 1}
|\det \tilde M|^{q} \, (\mathrm d \mathbf m_0^N) \cdots  (\mathrm d \mathbf m_N^{N}) \\ = {2 \over  \sigma_{\beta (n - N)} B({\beta \over 2}(n-N), 
{\beta \over 2} N(n+1) + 1)}
\prod_{i=1}^N\frac{\sigma_{\beta(N-i+1)}}{\sigma_{\beta(n-i+1)}}\,
\Big ( {  \sigma_{\beta n} \over   \sigma_{\beta N}} \Big )^{N +1} \bigg |_{n = N + q/\beta},
\end{multline}
and
\begin{multline}\label{37}
 \Big ( {1 \over \pi} \Big )^{\beta N(N+1)/2}
 \int_{\mathbf m_l^N \in (\mathbb F_\beta)^N}   e^{- \sum_{l=0}^N \| \mathbf m_l^N \|^2 }   |\det \tilde M|^{q} \,  \, (\mathrm d \mathbf m_0^N) \cdots  (\mathrm d \mathbf m_N^N) \\
 =  \Big ( {N+1 \over  \pi} \Big )^{\beta ( n - N)/2}
  \prod_{i=1}^N\frac{\sigma_{\beta(N-i+1)}}{\sigma_{\beta(n-i+1)}}\, 
\Big ( {1 \over \pi} \Big )^{\beta (N+1) (N - n)/2}  \bigg |_{n = N + q/\beta}.
\end{multline}
\end{proposition}

\begin{proof}
Consider the affine Blascke--Petkantschin decomposition of measure (\ref{19.3}). Multiply both sides by
$\prod_{i=0}^N p(\| \mathbf v_i^n \|)$, where $p(\| \mathbf v^n \|)$ is normalised to integrate to unity.
Making use of (\ref{18'}) and (\ref{19.4}) we see that then
\begin{multline}\label{38}
1 =   \prod_{i=1}^N\frac{\sigma_{\beta(n-i+1)}}{\sigma_{\beta(N-i+1)}}\, 
 \int_{\mathbf v_k^N \in (\mathbb F_\beta)^{N+1}}   (\mathrm d \mathbf v_0^N) \cdots  (\mathrm d \mathbf v_N^N) \,
 \int  \mathrm{d} S_{n-N}^{\perp,\beta} \\
 \times \prod_{i=0}^N p( (\| \mathbf v_i^N \|^2 +  \|  \mathbf r \|^2)^{1/2}) \,| \det  [ \mathbf v_k^N - \mathbf v_0^N ]_{k=1}^N |^{\beta (n- N)}.
 \end{multline}
 From the meaning of $\mathbf r$ given in the text below (\ref{19.2}), we have $\|  \mathbf r \| = r$, with $r$ appearing in the
 polar decomposition of $ S_{n-N}^{\perp,\beta} $ according to the infinitesimal formula
 \begin{equation}\label{39}
 d  S_{n-N}^{\perp,\beta}  = r^{\beta (n - N) - 1} \mathrm dr \, (U_1^\dagger d U_1),
 \end{equation}
 where $U_1 \in \mathcal V_{1,n-N}^\beta$. Substituting (\ref{39}) in (\ref{38}) we see that the integrand is independent of $U_1$.
 Integrating over this quantity by making use of (\ref{14}) with $N=1$, $n \mapsto n - N$, we read off from (\ref{38}) that
 \begin{multline}\label{40}
 \int_{\mathbf v_k^N \in (\mathbb F_\beta)^{N+1}}   (\mathrm d \mathbf v_0^N) \cdots  (\mathrm d \mathbf v_N^N) \,
 \int_0^\infty \mathrm dr \, r^{\beta ( n - N) - 1} \,
 \prod_{i=0}^N p( (\| \mathbf v_i^N \|^2 +  \|  \mathbf r \|^2)^{1/2}) \\
 \times  \,| \det  [ \mathbf v_k^N - \mathbf v_0^N ]_{k=1}^N|^{\beta (n- N)} 
 =  {1 \over \sigma_{\beta (n - N)}} \prod_{i=1}^N\frac{\sigma_{\beta(N-i+1)}}{\sigma_{\beta(n-i+1)}}.
 \end{multline}
 
 Let us now specify to the choice of $p(\| \mathbf v^n \|) $ to be uniform on $\| \mathbf v^n \| \le 1$, by choosing
 \begin{equation}\label{27}
{p}(x) = {\beta n \over \sigma_{\beta n}} \chi_{0 < x < 1},
\end{equation}
(cf.~(\ref{27})). The integration over $\mathbf v_k^N \in (\mathbb F_\beta)^{N+1}$ is now restricted to $\| \mathbf v_k^N \|^2 \le
1 - r^2$, and with this done $ \prod_{i=0}^N p( (\| \mathbf v_i^N \|^2 +  \|  \mathbf r \|^2)^{1/2})$ can be replaced by
$(\beta n/  \sigma_{\beta n})^{N+1}$. Now introduce the scaled variables $\mathbf m_k^N$ according to
$\mathbf v_k^N = (1 - r^2)^{1/2}\mathbf m_k^N $.  The LHS of (\ref{40}) then reads
 \begin{multline*}
\Big ( {\beta n \over  \sigma_{\beta n} }\Big )^{N+1}
 \int_{\| \mathbf m_k^N \| \le 1 }   (\mathrm d \mathbf m_0^N) \cdots  (\mathrm d \mathbf m_N^N) \, | \det \tilde M|^{\beta (n - N)} \\
 \times
 \int_0^1 \mathrm dr \, r^{\beta ( n - N) - 1} (1 - r^2)^{\beta N (n + 1)/2}. 
  \end{multline*}
 The integration over $r$ thus factorises and is furthermore, after a simple change of variables, an example of the Euler beta
 integral (\ref{Beta}). The result (\ref{36}) now follows. The derivation of (\ref{36'}) is very similar, starting with
 $p(x) = (1/ \sigma_{\beta n} ) \delta (x - 1)$. The key step is to change variables $\| \mathbf v_k^N \|^2 = s_k^2 - r^2$, which 
 allows the dependence on $s_k$ and $r$ to be integrated out.
 
 As our third and final choice of $p(\| \mathbf v^n \|)$ we set
 $$
 p(x) = \Big ( {1 \over \pi} \Big )^{\beta n /2} e^{- x^2}
 $$
 (cf.~(\ref{27b})). The LHS of (\ref{40}) then reads
  \begin{multline*}
  \Big ( {1 \over \pi} \Big )^{\beta n(N+1) /2}
   \int_{\mathbf m_k^N \in (\mathbb F_\beta)^{N+1}}   (\mathrm d \mathbf m_0^N) \cdots  (\mathrm d \mathbf m_N^N) \,
   e^{- \sum_{i=0}^N \| \mathbf m_i \|^2}  | \det \tilde M |^{\beta ( n - N)} \\ \times
 \int_0^\infty \mathrm dr \, r^{\beta ( n - N) - 1} e^{- (N+1) r^2},
 \end{multline*} 
 which implies (\ref{37}), after noting $2/(\sigma_{s} \Gamma(s)) = \pi^{-s}$.
 
\end{proof}

\begin{remark}
In relation to Proposition \ref{P7} we remarked that the result for the uniform distribution on $\| \mathbf m_i^N \| = 1$ could
be deduced from either of the results for the uniform distribution on $\| \mathbf m_i^N \| \le 1$ or the Gaussian distribution.
Notwithstanding the relationship between (\ref{35}) and (\ref{35'}) being analogous to that between
(\ref{36}) and (\ref{36'}), due to the different meaning of $M$ in  Proposition \ref{P8} we have not been able to deduce 
(\ref{36'}) as a corollary of the other results.
\end{remark}

\begin{remark}
Setting $q=\beta$ in (\ref{36}) and simplifying gives
 \begin{equation}
 \Big \langle | \det \tilde{M} |^\beta \Big \rangle_{\| \mathbf m_k^N \| \le 1} =
 {\Gamma({\beta \over 2}(N+1)^2 + 1) \Gamma({\beta \over 2}(N+1)) \over \Gamma({\beta \over 2}) \Gamma({\beta \over 2}N(N+2)+1)} 
\bigg ( { \Gamma({\beta \over 2}N+1) \over  \Gamma({\beta (N+1)\over 2}+1 } \bigg )^{N+1}.
\end{equation}
For $\beta = 1$ this can be expressed in terms of binomial coefficients according to 
 \begin{equation}
 {1 \over N! \sigma_N} \Big \langle | \det \tilde{M} | \Big \rangle_{\| \mathbf m_k^N \| \le 1} = 2^{-N}
 \binom{(N+1)}{(N+1)/2}^{N+1} \Big /  \binom{(N+1)^2}{(N+1)^2/2},
\end{equation}
which is a formula due to Kingman \cite{Ki69} and highlighted in \cite{Kl69} and \cite[End of Ch.~5]{So78}.

The substitution $q = \beta$ in the Gaussian case (\ref{37}) gives
 \begin{equation}
{1 \over N!}  \Big \langle | \det \tilde{M} |^\beta \Big \rangle_{(\mathbf m_{k}^{N})_{i,\beta} \in {\rm N}[0,1/\sqrt{2}] }  =
{(N+1)^{\beta/2} \over N!} {\Gamma({\beta \over 2}(N+1)) \over \Gamma({\beta\over 2})}.
 \end{equation}
 With $\beta = 1$ this is equivalent to a result contained in Efron \cite{Ef65}; see also the review section of 
 \cite{MCR10}.
 
\end{remark} 

Analogous to Corollary \ref{C3.8}, we can use knowledge of the averages in Proposition \ref{P8},
together with Corollary \ref{C3.6}, to specify moments of
$\det \tilde M^\dagger \tilde M$ when $\tilde{M} = [\mathbf m_k^n - \mathbf m_0^n]_{k=1}^N$, and thus of size $n \times N$.
This extends the working of  \cite{Mi71} in the real case.

\begin{corollary}\label{C3.10}
With $\tilde M$ specified as above, let $\tilde M_N = [\mathbf m_k^N - \mathbf m_0^N]_{k=1}^N$ correspond to the $N \times N$ case.
We have
\begin{align*}
\Big \langle ( \det \tilde M^\dagger \tilde M)^{h/2}  \Big  \rangle_{\| \mathbf m_k^n \| \le 1} & =
\Big \langle | \det \tilde M_N |^{h + \beta(n-N)}  \Big  \rangle_{\| \mathbf m_k^N \| \le 1}\Big /
 \Big \langle | \det \tilde M_N |^{\beta(n-N)}  \Big  \rangle_{\| \mathbf m_k^N \| \le 1} \\
\Big \langle ( \det \tilde M^\dagger \tilde M)^{h/2}  \Big  \rangle_{\| \mathbf m_k^n \| = 1} & =
\Big \langle | \det \tilde M_N |^{h + \beta(n-N)} \Big  \rangle_{\| \mathbf m_k^N \| = 1} \Big /
\Big \langle | \det \tilde M_N |^{\beta(n-N)} \Big  \rangle_{\| \mathbf m_k^N \| = 1}
\\
\Big \langle ( \det \tilde M^\dagger \tilde M)^{h/2} \Big \rangle_{{\rm G}^n} & =
 \Big \langle | \det \tilde M_N |^{h + \beta(n-N)} \Big \rangle_{{\rm G}^N}  \Big /
  \Big \langle | \det \tilde M_N |^{\beta(n-N)} \Big \rangle_{{\rm G}^N}
\end{align*}
where in the final equation the subscripts ${\rm G}^n$ and ${\rm G}^N$ refer to the Gaussian distribution on $\tilde M$ and $\tilde M_N$.
\end{corollary}

The simplest case is when $N=1$. Then $ \det \tilde M^\dagger \tilde M = \| \mathbf m_1^n - \mathbf m_0^n\|^2$, which is thus the
squared Euclidean length between the two vectors $ \mathbf m_1^n$ and $ \mathbf m_0^n$. This must be a function of the
combination $\beta n$, and $h$, only since the modulus squared of a complex and real quaternion number is equal to the square of $\beta$ real
numbers ($\beta = 2, 4$). Indeed we can check from Proposition \ref{P8} that this property holds true. For example, in the
Gaussian case with $N=1$
$$
\Big \langle ( \det \tilde M^\dagger \tilde M)^{h/2} \Big \rangle_{{\rm G}^n} = 2^{h/2} {\Gamma((\beta n + h)/2) \over
\Gamma(\beta n /2)}.
$$
Note from this that the  distribution of $ \| \mathbf m_1^n - \mathbf m_0^n\|^2$ is therefore equal to 
the Gamma distribution $\Gamma[\beta n /2 -1, 1/2]$. This can be anticipated, as with each independent real component
of $ \mathbf m_1^n$ and $ \mathbf m_0^n$ distributed as N$[0,1/\sqrt{2}]$, their difference is distributed as N[0,1], so the
problem reduces to asking for the sum of the square of $\beta n$ standard real Gaussians.

\section*{Acknowledgements}
This research is  part of the program of study supported by the 
ARC Centre of Excellence for Mathematical \& Statistical Frontiers.

%\bibliographystyle{amsplain}
%\bibliography{book1t16}

\begin{thebibliography}{10}

\bibitem{BN89}
O.E. Barndorff-Nielsen, P.~Blaesild, and P.S. Eriksen, Decomposition of
  invariant measure and statistical transformation models, Lecture notes in
  mathematics, vol.~58, Springer, Berlin, 1989.

\bibitem{Bl35}
W.~Blaschke, \emph{{I}ntegralgeometrie i. {E}rmittlung der {D}ichten f\"ur
  lineare {U}nterr\"aume im {$E_n$}}, Acta.~Sci.~Industr. \textbf{252} (1935),
  1--22.

\bibitem{DG09}
J.A. D\'iaz-Garc\'ia and R.~Guti\'errex-J\'aimez, \emph{Random matrix theory
  and multivariate statistics}, arXiv:0907.1064.

\bibitem{DG11}
\bysame, \emph{On {W}ishart distribution: some extensions}, Linear Alg.
  Applications \textbf{435} (2011), 1296--1310.

\bibitem{Ef65}
B.~Efron, \emph{The convex hull of a random set of points}, Biometrika
  \textbf{52} (1965), 331 -- 343.

\bibitem{FBKSZ12}
J.~Fischmann, W.~Bruzda, B.A. Khoruzhenko, H.-J. Sommers, and K.~Zyczkowski,
  \emph{Induced {G}inibre ensemble of random matrices and quantum operations},
  J. Phys. A \textbf{45} (2012), 075203.

\bibitem{Fo10}
P.J. Forrester, \emph{Log-gases and random matrices}, Princeton University
  Press, Princeton, NJ, 2010.
  
\bibitem{Fo13}
\bysame, \emph{Lyapunov exponents of products of complex {G}aussian random
  matrices}, J. Stat. Phys. \textbf{151} (2013), 796--808.

\bibitem{Fo15}
\bysame, \emph{Asymptotics of finite system Lyapunov exponents for some random matrix ensembles}, 
J. Phys. A \textbf{48} (2015), 215205.


\bibitem{Ja54}
A.T. James, \emph{Normal multivariate analysis and the orthogonal group}, Ann.
  Math. Statist. \textbf{25} (1954), 40--75.

\bibitem{Ki69}
J.F.C. Kingman, \emph{Random secants of a convex body}, J. Appl. Prob.
  \textbf{6} (1969), 660--672.

\bibitem{Kl69}
V.~Klee, \emph{What is the expected volume of a simplex whose vertices are
  chosen at random from a given convex body?}, Am. Math. Monthly \textbf{76}
  (1969), 286--288.

\bibitem{MCR10}
S.N. Majumdar, A.~Comtet, and J.~Randon-Furling, \emph{Random convex hulls and
  extreme value statistics}, J. Stat. Phys. \textbf{138} (2010), 955 -- 1009.
  
  \bibitem{Ma98}
A.M. Mathai, \emph{Random $p$-content of a $p$-parallelotope in {E}uclidean
$n$-space}, Adv. Appl. Probab. \textbf{31} (1999), 343--354.

\bibitem{Ma99}
\bysame, \emph{An introduction to geometric probability}, Gordon and Breach
  Science Publishers, Amsterdam, 1999.

\bibitem{Mi71}
R.E. Miles, \emph{Isotropic random simplifies}, Adv. Appl. Prob. \textbf{3}
  (1971), 353--382.

\bibitem{Mo12}
S.R. Moghadasi, \emph{Polar decomposition of the $k$-fold {L}ebesgue measure on
  {$\mathbb R^n$}}, Bull. Aust. Math. Soc. \textbf{85} (2012), 315--324.

\bibitem{Mu82}
R.J. Muirhead, \emph{Aspects of multivariate statistical theory}, Wiley, New
  York, 1982.
  
  \bibitem{Ne86}
C.M. Newman, \emph{The distribution of {L}yapunov exponents: exact results for
  random matrices}, Commun. Math. Phys. \textbf{103} (1986), 121--126.


\bibitem{Pe36}
B.~Petkantschin, \emph{{I}ntegralgeometrie 6. {Z}usammenh\"lange zwischen den
  {D}richten der linearen {U}nterr\"aume im $n$-dimensionalen {R}aum}, Abh.
  Math. Sem. Hamburg \textbf{11} (1936), 249--310.

\bibitem{Ro07}
A.~Rouault, \emph{Asymptotic behavior of random determinants in the {L}aguerre,
  {G}ram and {J}acobi ensembles}, Alea \textbf{3} (2007), 181--230.

\bibitem{RM79}
H.~Ruben and R.E. Miles, \emph{A canonical decomposition of the probability
  measure of sets of isotropic random points in {$\mathbb R^n$}}, J. Mult.
  Analysis \textbf{10} (1980), 1--18.
  
 \bibitem{Sc14}
 R.~Schneider, \emph{Convex bodies: the Brunn--Minkowski theory}, 2nd edition, CUP,
 Cambridge, 2014. 

\bibitem{So78}
H.~Solomon, \emph{An introduction to geometric probability}, SIAM Publications,
  Philadelphia, 1978.

\bibitem{Wi32}
S.S. Wilks, \emph{Certain generalizations in the analysis of variance},
  Biometrika \textbf{24} (1932), 471--494.

\bibitem{Wi28}
J.~Wishart, \emph{The generalized product moment distribution in samples from a
  normal multivariate population}, Biometrika \textbf{20A} (1928), 32--43.

\end{thebibliography}
\nopagebreak

\providecommand{\bysame}{\leavevmode\hbox to3em{\hrulefill}\thinspace}
\providecommand{\MR}{\relax\ifhmode\unskip\space\fi MR }
\providecommand{\MRhref}[2]{%
  \href{http://www.ams.org/mathscinet-getitem?mr=#1}{#2}
}
\providecommand{\href}[2]{#2}

\end{document}